\documentclass[11pt]{amsart}
\usepackage{amssymb,amsmath,amsthm,latexsym,stmaryrd,cancel}
\usepackage{xcolor}

\usepackage[T1]{fontenc}
\usepackage{lmodern}

\usepackage[pagebackref,colorlinks=true]{hyperref}  
\newtheorem*{thmA}{Theorem A}
\newtheorem*{thmB}{Theorem B}
\newtheorem{theorem}{Theorem}[section]

\newtheorem{corollary}[theorem]{Corollary}

\newtheorem{definition}[theorem]{Definition}

\newtheorem{proposition}[theorem]{Proposition}
\newtheorem{remark}[theorem]{Remark}

\newcommand{\ie}{i.\hspace{.5pt}e.\ }
\newcommand{\f}{\phi}
\newcommand{\g}{\tilde{g}}
\newcommand{\n}{\nabla}

\newcommand{\M}{(\mathcal{M},\A\f,\A\xi,\A\eta,\A{}g)}

\newcommand{\R}{\mathbb R}

\newcommand{\X}{\mathfrak X}
\newcommand{\F}{\mathcal{F}}
\newcommand{\LL}{\mathcal{L}}
\newcommand{\ta}{\theta}
\newcommand{\om}{\omega}
\newcommand{\lm}{\lambda}

\newcommand{\al}{\alpha}

\DeclareMathOperator{\Div}{div} 
\DeclareMathOperator{\tr}{tr} 
\DeclareMathOperator{\Span}{span} 

\newcommand{\corref}[1]{Corollary~\ref{#1}}
\newcommand{\propref}[1]{Proposition~\ref{#1}}
\newcommand{\remref}[1]{Remark~\ref{#1}}

\newcommand{\A}{\allowbreak{}}

\begin{document}

\title[Para-Ricci-like Solitons on ApapR Manifolds]
{Para-Ricci-like Solitons on Almost Paracontact Almost Paracomplex Riemannian Manifolds}

\author[H. Manev]{Hristo Manev}
\author[M. Manev]{Mancho Manev}

\address[H. Manev]{Medical University of Plovdiv, Faculty of Public Health,
Department of Medical Informatics, Biostatistics and E-Learning, 15A Vasil Aprilov Blvd.,
Plovdiv 4002, Bulgaria}
\email{hristo.manev@mu-plovdiv.bg}
\address[M. Manev]{University of Plovdiv Paisii Hilendarski,
Faculty of Mathematics and Informatics, Department of Algebra and
Geometry, 24 Tzar Asen St., Plovdiv 4000, Bulgaria
\&
Medical University of Plovdiv, Faculty of Public Health,
Department of Medical Informatics, Biostatistics and E-Learning, 15A Vasil Aprilov Blvd.,
Plovdiv 4002, Bulgaria}
\email{mmanev@uni-plovdiv.bg}

\begin{abstract}
It is introduced and studied para-Ricci-like solitons with potential Reeb vector field
on almost paracontact almost paracomplex Riemannian manifolds.
The special cases of para-Einstein-like, para-Sasaki-like and having a torse-forming Reeb vector field have been considered. It is proved a necessary and sufficient condition the manifold to admit a para-Ricci-like soliton which is the structure to be para-Einstein-like.
{Explicit examples are provided in support of the proven statements.}
\end{abstract}




\maketitle



\section{Introduction}

The concept of Ricci solitons is introduced in 1982 by R.\,S. Hamilton \cite{Ham82} as a special self-similar solution of the Ricci flow equation, and it plays an important role in understanding its singularities.
In other words, Ricci solitons represent a generalization of Einstein metrics on a Riemannian
manifold, being generalized fixed points of the Ricci flow, considered as a dynamical system.
A detailed study on Riemannian Ricci solitons may be found in \cite{Cao}.

{The study of Ricci solitons in contact Riemannian geometry began in 2008 by R. Sharma \cite{Shar}.}
After that the investigation of the geometric characteristics of the Ricci solitons continues on different {types} of almost contact metric manifolds (see \cite{IngBag}, \cite{BagIng13}, \cite{NagPre}, \cite{AyaYil}, \cite{GalCra}).
The notion of $\eta$-Ricci soliton, introduced by J.\,T. Cho and M. Kimura \cite{ChoKim}, is a generalization of the concept of Ricci soliton. Investigations on $\eta$-Ricci solitons have been made in \cite{Bla15}, \cite{BlaCra}, \cite{PraHad} but in the context of paracontact geometry.

Recently, {various} geometers {have been extensively studying} the pseudo-Riemannian properties of Ricci solitons as its generalizations. This interest is based on the growing interest of theoretical physicists {in} Ricci solitons and their relation with string theory. For further references on pseudo-Riemannian Ricci solitons, see \cite{CalPer}, \cite{Bro-etal}, \cite{BagIng12}, \cite{BlaPer}, \cite{MM-Sol1}.

The object of our focus is the geometry of the almost paracontact almost paracomplex Riemannian manifolds. In these manifolds, the induced almost product structure on the paracontact distribution is traceless and the restriction on the paracontact distribution of the almost paracontact structure is an almost paracomplex structure. In \cite{ManSta}, these manifolds are introduced and classified. Their investigation continues in \cite{ManTav57,ManTav2}.

In the present paper, due to the presence of two associated metrics {each other}, we introduce and study our generalization of the Ricci soliton compatible with the
{manifold structure}, called para-Ricci-like soliton. We characterize it in terms of some important {types} of manifolds under consideration -- para-Einstein-like, para-Sasaki-like and having a torse-forming Reeb vector field.
We comment two examples in {support} of the {proven} statements.

\section{The main results}

\begin{thmA}\label{thm:RlSl}
Let $\M$ be a $(2n+1)$-dimensional para-Sasaki-like almost paracontact almost paracomplex Riemannian manifold. Let $a$, $b$, $c$, $\lm$, $\mu$, $\nu$ be constants {satisfying the conditions:}
\begin{equation}\label{SlElRl-const}
a+\lm=0,\qquad b+\mu+1=0,\qquad c+\nu-1=0.
\end{equation}
Then, $\M$ admits
a para-Ricci-like soliton with potential $\xi$ and constants $(\lm,\A\mu,\A\nu)$, where $\lm+\mu+\nu=2n$,
if and only if
it is para-Einstein-like with constants $(a,b,c)$, where $a+b+c=-2n$.

In particular, we get:
\begin{enumerate}
	\item[(i)] $\M$ admits an $\eta$-Ricci soliton with potential $\xi$ and constants $(\lm,0,2n-\lm)$ if and only if $\M$ is para-Einstein-like with constants $(-\lm,-1,\lm-2n+1)$.

	\item[(ii)] $\M$ admits a shrinking Ricci soliton with potential $\xi$ and constants $(2n,0,0)$ if and only if $\M$ is para-Einstein-like with constants $(-2n,-1,1)$.

	\item[(iii)] $\M$ is $\eta$-Einstein with constants $(a,0,-2n-a)$  if and only if
$\M$ admits a para-Ricci-like soliton with potential $\xi$ and constants $(-a,-1,a+2n+1)$.

	\item[(iv)] $\M$ is Einstein with constants $(2n,0,0)$ if and only if
$\M$ admits a para-Ricci-like soliton with potential $\xi$ and constants $(2n,-1,1)$.
\end{enumerate}
\end{thmA}

\begin{thmB}\label{thm:Rltf}
Let {$\M$ be an almost paracontact almost paracomplex Riemannian manifold and $\xi$ be} torse-forming with function $f$. Then, $\M$ admits a para-Ricci-like soliton with potential $\xi$ and constants $(\lm,\mu,\nu)$ if and only if $\M$ is para-Einstein-like with constants $(a,b,c)$ {satisfying the conditions:}
\begin{equation}\label{tfElRl-const}
a+\lm+f=0,\qquad b+\mu=0,\qquad c+\nu-f=0,
\end{equation}
where $f$ is a constant.

In particular, we have:
\begin{enumerate}
	\item[(i)] $\M$ admits an $\eta$-Ricci soliton with potential $\xi$ and constants $(\lm,0,\nu)$ if and only if $\M$ is $\eta$-Einstein with constants $(a,0,c)$, where $a+c=-\lm-\nu$.

	\item[(ii)] $\M$ admits a Ricci soliton with potential $\xi$ and constants $(\lm,0,0)$ if and only if
$\M$ is $\eta$-Einstein with constants $(-\lm-f,0,f)$.

	\item[(iii)] $\M$ is Einstein with constants $(a,0,0)$ if and only if $\M$ admits an $\eta$-Ricci soliton with potential $\xi$ and constants $(-a-f,0,f)$.
\end{enumerate}
\end{thmB}

\section{Almost paracontact almost paracomplex Riemannian manifolds}

Let $\M$ be an \emph{almost paracontact almost paracomplex Riemannian manifold} (briefly, {an} apapR manifold){. Namely,}  $\mathcal{M}$ is a differentiable odd-dimensional manifold, {equipped with a Rie\-mannian metric $g$ and an almost paracontact structure $(\f,\xi,\eta)$, \ie} $\f$ is a (1,1)-tensor field, $\xi$ is the characteristic vector field and $\eta$ is its dual 1-form, which satisfy the following:
\begin{equation}\label{strM}
\begin{array}{c}
\f\xi = 0,\qquad \f^2 = I - \eta \otimes \xi,\qquad
\eta\circ\f=0,\qquad \eta(\xi)=1,\\ \
\tr \f=0,\qquad g(\f x, \f y) = g(x,y) - \eta(x)\eta(y),
\end{array}
\end{equation}
where $I$ is the identity transformation on $T\mathcal{M}$ (\cite{Sato76}, \cite{ManTav57}). Consequently, we get the following equations:
\begin{equation}\label{strM2}
\begin{array}{ll}
g(\f x, y) = g(x,\f y),\qquad &g(x, \xi) = \eta(x),
\\
g(\xi, \xi) = 1,\qquad &\eta(\n_x \xi) = 0,
\end{array}
\end{equation}
where {$\n$ denotes} the Levi-Civita connection of $g$.
Here and further $x$, $y$, $z$, $w$ stand for arbitrary vector fields from $\X(\mathcal{M})$ or vectors in $T\mathcal{M}$ at a fixed point of $\mathcal{M}$.

The associated metric $\g$ of $g$ on $\M$ is determined by {the equality} $\g(x,y)=g(x,\f y)+\eta(x)\eta(y)$. Therefore $\g$ is {an} indefinite metric of signature $(n + 1, n)$, which is compatible with $\M$ {in the same way} as $g$.

The apapR manifolds are classified in \cite{ManSta}. This classification contains eleven basic classes $\F_1$, $\F_2$, $\dots$, $\F_{11}$ and it is made with respect to the (0,3)-tensor $F$ determined by
\begin{equation}\label{F=nfi}
F(x,y,z)=g\bigl( \left( \nabla_x \f \right)y,z\bigr).
\end{equation}
The following basic properties of $F$ are valid:
\begin{equation}\label{F-prop}
\begin{array}{l}
F(x,y,z)=F(x,z,y)\\
\phantom{F(x,y,z)}
=-F(x,\f y,\f z)+\eta(y)F(x,\xi,z)
+\eta(z)F(x,y,\xi),\\
(\n_x\eta)y=g(\n_x\xi,y)=-F(x,\f y, \xi).
\end{array}
\end{equation}

{The intersection of the basic classes is denoted by $\F_0$,} which is the special class determined by the condition $F=0$.

The 1-forms associated with $F$ (also known as Lee forms) are defined by:
\begin{equation}\label{t}
\theta=g^{ij}F(e_i,e_j,\cdot),\quad
\theta^*=g^{ij}F(e_i,\f e_j,\cdot), \quad \omega=F(\xi,\xi,\cdot),
\end{equation}
where $\left(g^{ij}\right)$ is the inverse matrix of $\left(g_{ij}\right)$ of $g$ with respect to
a basis $\left\{\xi;e_i\right\}$ $(i=1,2,\dots,2n)$ of $T_p\mathcal{M}${, $p\in \mathcal{M}$}.

\subsection{{Para-Sasaki-like manifolds}}

In \cite{IvMaMa2}, {the class} of \emph{para-Sasaki-like spa\-ces} {is defined} in the set of almost paracontact almost paracomplex Riemannian manifolds {obtained} from a {specific} cone construction. The considered para-Sasaki-like spaces are determined by
\begin{equation}\label{defSl}
\begin{array}{l}
\left(\nabla_x\f\right)y=-g(x,y)\xi-\eta(y)x+2\eta(x)\eta(y)\xi,\\
\phantom{\left(\nabla_x\f\right)y}=-g(\f x,\f y)\xi-\eta(y)\f^2 x.
\end{array}
\end{equation}

\begin{remark}\label{rem:Sl}
The apapR manifolds of para-Sasaki-like type {form} a subclass of the basic class $\F_4$ {from} \cite{ManSta} and {their Lee forms are} $\ta=-2n\,\eta$ {and} $\ta^*=\om=0$.
\end{remark}

In \cite{IvMaMa2}, the truthfulness of the following identities is proved
\begin{equation}\label{curSl}
\begin{array}{ll}
\n_x \xi=\f x, \qquad &\left(\n_x \eta \right)(y)=g(x,\f y),\\
R(x,y)\xi=-\eta(y)x+\eta(x)y, \qquad &R(\xi,y)\xi=\f^2y, \\
\rho(x,\xi)=-2n\, \eta(x),\qquad 				&\rho(\xi,\xi)=-2n,
\end{array}
\end{equation}
where {$R$ and $\rho$ denote} the curvature tensor and the Ricci tensor, respectively.

\subsection{{Para-Einstein-like manifolds}}

\begin{definition}
An apapR manifold $\M$ is said to be
\emph{para-Ein\-stein-like} with constants $(a,b,c)$ if its Ricci tensor $\rho$ satisfies
\begin{equation}\label{defEl}
\begin{array}{l}
\rho=a\,g +b\,\g +c\,\eta\otimes\eta.
\end{array}
\end{equation}
\end{definition}

In particular, {the manifold is known to be} called an \emph{$\eta$-Einstein manifold} and an \emph{Einstein manifold} when $b=0$ and $b=c=0$, respectively.

Using \eqref{defEl}, we get the corresponding scalar curvature
\begin{equation}\label{tauEl}
\begin{array}{l}
\tau=(2n+1)a +b +c.
\end{array}
\end{equation}

\begin{remark}
{The Reeb vector field $\xi$ of a para-Einstein-like manifold $\M$
is an eigenvector of the Ricci operator $Q=\rho^{\sharp}$, \ie $g(Qx,y)=\rho(x,y)$,
corresponding to the eigenvalue $a+b+c$.}
\end{remark}

\begin{proposition}
The Ricci tensor $\rho$ of a para-Einstein-like {manifold} $\M$ has the following properties:
\begin{equation}\label{roEl1}
\begin{array}{l}
\begin{array}{l}
\rho(\f x,\f y)=-\rho(x,y)-(a+b+c)\eta(x)\eta(y),
\end{array}\\
\begin{array}{ll}
\rho(\f x,y)=\rho(x,\f y), \qquad  & \rho(\f x,\xi)=0,
\\
\rho(x,\xi)=(a+b+c)\eta(x), \qquad  & \rho(\xi,\xi)=a+b+c,
\end{array}
\end{array}
\end{equation}
\begin{equation}\label{roEl2}
\begin{array}{l}
\left(\n_x\rho\right)(y,z)=b\,g\bigl(\left(\n_x\f\right)y,z\bigr) \\
\phantom{\left(\n_x\rho\right)(y,z)=}
+(b+c)\left\{g(\n_x \xi,y)\eta(z) +g(\n_x \xi,z)\eta(y)\right\}.
\end{array}
\end{equation}
\end{proposition}
\begin{proof}
We confirm the {validity} of the proposition applying \eqref{strM} and \eqref{strM2} in \eqref{defEl}.
\end{proof}

\begin{proposition}
Let $\M$ be a $(2n+1)$-dimensional apapR para-Sasaki-like para-Einstein-like manifold with constants $(a,b,c)$.
Then we have
\begin{equation}\label{tauElSl2}
a+b+c=-2n,\qquad \tau=2n(a-1),
\end{equation}
\begin{equation}\label{roElSl}
\begin{array}{l}
\left(\n_x\rho\right)(y,z)=(b+c)\left\{g(x,\f y)\eta(z) +g(x,\f z)\eta(y)\right\}\\
\phantom{\left(\n_x\rho\right)(y,z)=}{}
-b\left\{g(\f x,\f y)\eta(z) +g(\f x,\f z)\eta(y)\right\},
\end{array}
\end{equation}
\begin{equation}\label{roElSl3}
\begin{array}{c}
a=-2n-\frac{1}{2n}(\Div^*{\rho})(\xi),\qquad
b=-\frac{1}{2n}(\Div{\rho})(\xi),\\[6pt]
c=\frac{1}{2n}\bigl\{(\Div{\rho})(\xi)+(\Div^*{\rho})(\xi)\bigr\},
\end{array}
\end{equation}
where $\Div$ and $\Div^*$ stand for the divergence with respect to $g$ and $\g$, respectively.
\end{proposition}
\begin{proof}
Using \eqref{tauEl} and the last equalities of \eqref{curSl} and \eqref{roEl1}, we obtain \eqref{tauElSl2}.
The next expression \eqref{roElSl} follows from \eqref{defSl} and \eqref{roEl2}.
We obtain the assertions in \eqref{roElSl3} by taking the traces over $x$ and $y$ in \eqref{roElSl} with respect to both metrics and using that $\g^{ij}=\f^j_kg^{ik}+\xi^i\xi^j$.
\end{proof}

\begin{corollary}\label{cor:ElSl}
Let $\M$ be a para-Sasaki-like para-Ein\-stein-like manifold with constants $(a,b,c)$. Then:
\begin{enumerate}
\item[(i)] A necessary and sufficient condition $\M$ to be scalar-flat is $a=1$.

\item[(ii)] A necessary and sufficient condition $\M$ to be Ricci-sym\-met\-ric is to be an Einstein manifold.

\item[(iii)] The Ricci tensor of $\M$ is $\eta$-parallel and parallel along $\xi$.

\item[(iv)] A necessary and sufficient condition $\M$ to be $\eta$-Einstein is $\Div Q\in\ker\eta$.

\item[(v)] A necessary and sufficient condition $\M$ to be Einstein is $\Div{Q}, \Div^*{Q}\in\ker\eta$.

\end{enumerate}
\end{corollary}
\begin{proof}
It is easy to check the {validity} of (i) and (ii), bearing in mind \eqref{tauElSl2} and \eqref{roElSl}, respectively.
The truthfulness of (iii) comes from $\left(\n\rho\right)|_{\ker\eta}=0$ and  $\n_\xi\rho=0$, which are corollaries of \eqref{roElSl}.
{Similarly, we obtain} that (iv) and (v) are valid using the value of $b$ in \eqref{roElSl3} and the identities
$\left(\n_x \rho\right)(y,z)=g\left(\left(\n_x Q\right)y,z\right)$
and
$(\Div\rho)(z)=g(\Div Q,z)$.
\end{proof}

Moreover, {using \eqref{tauElSl2},} we obtain that {each} Einstein para-Sasaki-like $\M$ has a scalar curvature $\tau=-2n(2n+1)$.

\subsection{{ApapR} manifolds with a torse-forming Reeb vector field}\label{sec:tf}

It is known that a vector field $\xi$ is called \emph{torse-forming} if $\n_x \xi=f\,x+\al(x)\xi$ {is valid} for a smooth function $f$ and an 1-form $\al$ on the manifold. The 1-form $\al$ on an apapR manifold $\M$ is determined by $\al=-f\,\eta$, {taking into account} the last equality in \eqref{strM2}. Then, we have the following equivalent equalities
\begin{equation}\label{tf}
\begin{array}{l}
		\n_x \xi=f\,\f^2x,\qquad \left(\n_x \eta \right)(y)=f\,g(\f x,\f y).
		\end{array}
\end{equation}

In the further investigations we omit the trivial case when $f = 0$.

\begin{proposition}\label{rem:tf}
Each $\M$ with torse-forming $\xi$ belongs to
the class $\F_1\oplus\F_2\oplus\F_3\oplus\F_5\oplus\F_6\oplus\F_{10}$.
Moreover, {among the basic classes, only} $\F_5$ can contain such a manifold.
\end{proposition}
\begin{proof}
Taking into account \eqref{F-prop} and \eqref{t}, {equalities} \eqref{tf} imply $\ta^*(\xi)=-2n\,f$ and $\ta(\xi)=\om=0$.
Therefore, bearing in mind the components of $F$ in the basic classes $\F_i$, given in \cite{ManTav57},
we get the {statement}.
\end{proof}

{As a result of the latter proposition, a manifold} $\M\in\F_5$ with torse-forming $\xi$ is {determined by}
\begin{equation}\label{tfF5}
\begin{array}{l}
		\left(\n_x \f \right)y=-f\{g(x,\f y)\xi+\eta(y)\f x\}.
\end{array}
\end{equation}

\begin{proposition}
{The Reeb vector field $\xi$ of each} para-Sasaki-like manifold $\M$ is not {torse-forming}.
\end{proposition}
\begin{proof}
Taking into account \remref{rem:Sl} and \propref{rem:tf}, we prove the assertion.
\end{proof}

\begin{corollary} Let an apapR manifold $\M$ with torse-forming $\xi$ be para-Einstein-like.
Then, a necessary and sufficient condition $\M$ to be Ricci-symmetric is to be an Einstein manifold.
\end{corollary}
\begin{proof}
{Similarly to (ii) {of} \corref{cor:ElSl}, we obtain the statement,} using \eqref{roEl2} and \eqref{tf}.
\end{proof}

\section{Proofs of Theorem A and Theorem B}\label{sect-1}
%

A \emph{Ricci soliton} is a pseudo-Riemannian manifold $(M,g)$ which admits a smooth non-zero vector field $v$ on $M$ such that \cite{Ham82}
\begin{equation*}\label{Rs}
\begin{array}{l}
\rho =-\frac12 \mathcal{L}_v g - \lm\, g,
\end{array}
\end{equation*}
where $\mathcal{L}$ stands for the Lie derivative and $\lm$ is a constant.
In particular, a Ricci soliton with negative, zero or positive $\lm$ is called \emph{shrinking}, \emph{steady} or \emph{expanding}, respectively \cite{ChoLuNi}.

The authors of \cite{ChoKim} made a generalization of the Ricci soliton on an almost contact metric manifold $(M,\varphi,\xi,\eta,g)$ called an \emph{$\eta$-Ricci soliton} determined by
$
\rho=-\frac12 \mathcal{L}_v g  - \lm\, g - \nu\, \eta\otimes\eta$,
where $\nu$ is also a constant.

{Due to the presence of two associated metrics on almost contact B-metric manifolds, a subsequent} generalization of the Ricci soliton and the $\eta$-Ricci soliton {is presented in \cite{MM-Sol1},} called a \emph{Ricci-like soliton}. Here, we introduce our generalization in terms of apapR manifolds as follows.
\begin{definition}
An {apapR} manifold $\M$ is called a
\emph{para-Ricci-like soliton} with potential vector field $\xi$ and constants $(\lm,\mu,\nu)$ if its Ricci tensor $\rho$ satisfies:
\begin{equation}\label{defRl}
\begin{array}{l}
\rho=-\frac12 \mathcal{L}_{\xi} g - \lm\, g - \mu\, \g - \nu\, \eta\otimes \eta.
\end{array}
\end{equation}
\end{definition}

Taking into account \eqref{defRl} and
\begin{equation}\label{Lg=nxi}
\left(\mathcal{L}_{\xi} g \right)(x,y)=g(\n_x \xi,y)+g(x,\n_y \xi),
\end{equation}
we obtain the scalar curvature
\[
\tau=-\Div\xi-(2n+1)\lm -\mu-\nu.
\]

\begin{proposition}\label{prop-RlEl}
Let $\M$ be a para-Einstein-like manifold with constants $(a,b,c)$ and let $\M$ admit a para-Ricci-like soliton with potential $\xi$ and constants $(\lm,\mu,\nu)$.
Then:
\begin{enumerate}
	\item[(i)] $a+b+c=-\lm-\mu-\nu$;

	\item[(ii)] $\n_{\xi}\xi=0$, \ie $\xi$ is geodesic;

	\item[(iii)] $\left(\n_{\xi} \f\right)\xi=0$, $\n_{\xi} \eta=0$, $\om=0$;

	\item[(iv)] $\left(\n_{\xi}\rho\right)(y,z)=b\,g\bigl(\left(\n_{\xi}\f\right)y,z\bigr)$.

\end{enumerate}
\end{proposition}
\begin{proof}
By virtue of \eqref{defEl}, \eqref{defRl} and \eqref{Lg=nxi}, we get
\begin{equation*}\label{ElRl}
\begin{array}{l}
g(\n_x \xi,y)+g(\n_y \xi,x)=-2\{(a+\lm)g(x,y)+(b+\mu)g(x,\f y)\\
\phantom{g(\n_x \xi,y)+g(\n_y \xi,x)=-2\{}
+(b+c+\mu+\nu)\eta(x)\eta(y)\}.
\end{array}
\end{equation*}
By substituting $y$ for $\xi$ in the latter equality, we prove the truthfulness of (i) and (ii).
The validity of (iii) comes from \eqref{F=nfi}, \eqref{F-prop} and \eqref{t}.
Finally, \eqref{roEl2} and $\n_{\xi} \xi=0$ yield (iv).
\end{proof}

\begin{corollary}
Let $\M$ be a para-Einstein-like manifold with constants $(a,b,c)$ and let $\M$ admit a para-Ricci-like soliton with potential $\xi$ and constants $(\lm,\mu,\nu)$. Then:
\begin{enumerate}
	\item[(i)] $\M$ does not belong to $\F_{11}$ or to {a direct sum of $\F_{11}$} with other basic classes.
	\item[(ii)] A necessary and sufficient condition the Ricci tensor of $\M$ to be parallel along $\xi$ is $\M$ not to belong to $\F_{10}$ or to {a direct sum of $\F_{10}$} with other basic classes.
	\item[(iii)] A necessary and sufficient condition the Ricci tensor of $\M$ to be parallel along $\xi$ is $\M$ to be $\eta$-Einstein.
\end{enumerate}
\end{corollary}
\begin{proof}
The characterization of the basic classes $\F_i$ by the components of $F$, given in \cite{ManTav57}, and the assertions (iii) and (iv) in \propref{prop-RlEl}, complete the proof of the corollary.
\end{proof}

{Using the expression of $\n_x \xi$ from \eqref{curSl} in} \eqref{Lg=nxi} and \eqref{defRl}, we obtain that \eqref{defRl} {takes the form}
\begin{equation}\label{SlRl-rho}
\rho=-\lm g-(1+\mu)\g+(1-\nu)\eta\otimes\eta,
\end{equation}
which coincides with \eqref{defEl} under conditions \eqref{SlElRl-const} in Theorem A.
The equality $a+b+c=-2n$ comes from \eqref{tauElSl2}, whereas
$\lm+\mu+\nu=2n$ is a result of \eqref{SlRl-rho} for $\rho(x,\xi)$ and the corresponding formula from \eqref{curSl}.
Therefore, the main assertion in Theorem A is proved.

We get the truthfulness of the assertions (i), (ii), (iii) and (iv) in Theorem A as corollaries of the main statement for the cases $\mu=0$, $\mu=\nu=0$, $b=0$ and $b=c=0$, respectively.
This completes the proof of Theorem A.

Now we focus our considerations on Theorem B.
Let $\xi$ on an apapR manifold $\M$ be torse-forming with function $f$ and let $\M$ admit a para-Ricci-like soliton with potential $\xi$ and constants $(\lm,\mu,\nu)$.
Taking into account \eqref{tf}, \eqref{defRl} and \eqref{Lg=nxi}, we get the form of the Ricci tensor of $\M$ as follows
\begin{equation*}\label{tfRl-rho}
\rho=-(\lm+f) g-\mu\g-(\nu-f)\eta\otimes\eta.
\end{equation*}
{Bearing} in mind \eqref{defEl}, the latter identity shows that a necessary and sufficient condition $\M$ to be para-Einstein-like is $f$ to be a constant. Then, the main statements \eqref{tfElRl-const} in Theorem B are valid.

Vice versa, let $\M$ have a torse-forming {vector field} $\xi$ with function $f$ and let $\M$ be para-Einstein-like with constants $(a,b,c)$. Applying \eqref{tf} and \eqref{defEl} in \eqref{Lg=nxi} and \eqref{defRl} for $(\lm,\mu,\nu)$, we get that \eqref{defRl} is satisfied if and only if \eqref{tfElRl-const} holds. This is true when $f$ is a constant.

We get the truthfulness of the assertions (i), (ii) and (iii) in Theorem B as corollaries of the main statement for the cases $b=\mu=0$, $b=\mu=\nu=0$ and  $b=c=\mu=0$, respectively.
This completes the proof of Theorem B.

\section{Some consequences of the main theorems}\label{sect-1}

\begin{proposition}\label{prop-RlEltf}
Let $\M$ be a para-Einstein-like manifold with constants $(a,b,c)$ and let $\M$ admits a para-Ricci-like soliton with potential $\xi$ and constants $(\lm,\mu,\nu)$.
Then $f$ is determined by:
\begin{equation}\label{tf1}
f=\varepsilon\sqrt{-\frac{a+b+c}{2n}}=\varepsilon\sqrt{\frac{\lm+\mu+\nu}{2n}}, \qquad \varepsilon=\pm 1.
\end{equation}
\end{proposition}

\begin{proof}
Using the form of $\n\xi$ in \eqref{tf}, we get
\begin{equation}\label{Rxif}
R(x,y)\xi=f^2\{\eta(x)y-\eta(y)x\},\qquad \rho(x,\xi)=2nf^2\eta(x).
\end{equation}
We get that $-2n f^2=a+b+c$ from \eqref{Rxif} and \eqref{roEl1}. This result and (i) of \propref{prop-RlEl} imply \eqref{tf1}.
\end{proof}

\begin{corollary}
The $\xi$-sections $\al=\Span\{x,\xi\}$ of an apapR manifold $\M$ with a torse-forming potential $\xi$ and constant $f$ have constant negative sectional curvatures $k(\al)=-f^2$.
\end{corollary}
\begin{proof}
The form of the curvature tensor $R$ in \eqref{Rxif} completes the proof of the corollary.
\end{proof}

\begin{corollary}\label{cor:E}
A necessary and sufficient condition $\M$ to be Einstein with a negative scalar curvature $\tau$ is $\M$ to admit an $\eta$-Ricci soliton with potential $\xi$ and constants
$(\lm=-\tau/(2n+1)-f, \mu=0, \nu=f),$
where $f$ is the constant for the torse-forming vector field $\xi$
and $f=\varepsilon\sqrt{-\tau/(2n(2n+1))}$, $\varepsilon=\pm 1$.
\end{corollary}
\begin{proof}
The expression of $\tau$ in \eqref{tauEl} for an Einstein manifold, \propref{prop-RlEltf} and assertion (iii) of Theorem B complete the proof of this corollary.
\end{proof}

{By virtue of}  \propref{rem:tf}, we conclude the following
\begin{proposition}\label{prop-RlEltfRs}
Let $\M\in\F_5\setminus\F_0$ be a para-Einstein-like manifold with constants $(a,b,c)$ and let $\M$ admit a para-Ricci-like soliton with potential $\xi$ and constants $(\lm,\mu,\nu)$. Then:
\begin{enumerate}
\item[(i)]  A necessary and sufficient condition $\M$ to be Ricci-sym\-met\-ric is $\M$ to be an Einstein manifold.
\item[(ii)] The Ricci tensor of $\M$ is $\eta$-parallel and parallel along $\xi$.
\end{enumerate}
\end{proposition}

\begin{proof}
Taking into account \eqref{roEl2}, \eqref{tf}, \eqref{tfF5} and \eqref{tfElRl-const}, we obtain
\begin{equation*}\label{roElF5}
\begin{array}{l}
\left(\n_x\rho\right)(y,z)=(a+\lm)\bigl\{(b+c)\left\{g(\f x, \f y)\eta(z) +g(\f x, \f z)\eta(y)\right\} \\
\phantom{\left(\n_x\rho\right)(y,z)=(a+\lm)\bigl\{(b}
+b\left\{g(x,\f y)\eta(z) +g(x,\f z)\eta(y)\right\}\bigr\}.
\end{array}
\end{equation*}
Using the latter result and following the proof for (ii) and (iii) in \corref{cor:ElSl}, we establish the truthfulness of the statements.
\end{proof}

\section{Examples}\label{examples}

\subsection{Example 1}
In \cite{IvMaMa2}, {a $5$-dimensional Lie group $G$} is considered with a basis of left-invariant vector fields $\{e_0,\dots, e_{4}\}$ {and the}  corresponding Lie algebra is determined {for $p,q\in\R$} as follows
\begin{equation}\label{comEx1}
\begin{array}{ll}
[e_0,e_1] = p e_2 - e_3 + q e_4,\quad &[e_0,e_2] = - p e_1 - q e_3 - e_4,\\[0pt]
[e_0,e_3] = - e_1  + q e_2 + p e_4,\quad &[e_0,e_4] = - q e_1 - e_2 - p e_3.
\end{array}
\end{equation}
In the {same} paper, $G$ is equipped with an invariant apapR structure $(\phi, \xi, \eta, g)$ {as follows:}
\begin{equation}\label{strEx1}
\begin{array}{l}
g(e_0,e_0)=g(e_1,e_1)=g(e_2,e_2)=g(e_{3},e_{3})=g(e_{4},e_{4})=1, \\[0pt]
g(e_i,e_j)=0,\quad i,j\in\{0,1,\dots,4\},\; i\neq j, \\[0pt]
\xi=e_0, \quad \f  e_1=e_{3},\quad  \f e_2=e_{4},\quad \f  e_3=e_{1},\quad \f  e_4=e_{2}.
\end{array}
\end{equation}
It is proved that the constructed manifold $(G, \phi, \xi, \eta, g)$ is a para-Sasaki-like manifold.
Taking into account \eqref{comEx1} and \eqref{strEx1}, we compute the components $R_{ijkl}=R(e_i,e_j,e_k,e_l)$ of the curvature tensor and the components   $\rho_{ij}=\rho(e_i,e_j)$ of the Ricci tensor.
The non-zero of them are determined by the following equalities and the well-known their symmetries and antisymmetries:
\begin{equation}\label{Rex1}
\begin{array}{l}
R_{0110}=R_{0220}=R_{0330}=R_{0440}=-1,\\
R_{1234}=R_{1432}=R_{1331}=R_{2442}=1,\qquad \rho_{00}=-4.
\end{array}
\end{equation}

Therefore, using \eqref{defEl}, we establish that the constructed manifold is $\eta$-Einstein with constants
\begin{equation}\label{abcS}
(a,b,c)=(0,0,-4).
\end{equation}

Now, we check whether $(G, \phi, \xi, \eta, g)$ admits a Ricci-like soliton with potential $\xi$.
By virtue of \eqref{strEx1} and \eqref{Lg=nxi},
we obtain the components $\left(\LL_\xi g\right)_{ij}=\left(\LL_\xi g\right)(e_i,e_j)$ and
the non-zero of them are
\begin{equation}\label{Lex1}
\begin{array}{l}
\left(\LL_\xi g\right)_{13}=\left(\LL_\xi g\right)_{24}=\left(\LL_\xi g\right)_{31}=\left(\LL_\xi g\right)_{42}=2.
\end{array}
\end{equation}
Afterwards, using \eqref{strEx1}, \eqref{Rex1} and \eqref{Lex1}, we get that \eqref{defRl} is satisfied for
\begin{equation}\label{lmnS}
(\lm,\mu,\nu)=(0,-1,5).
\end{equation}
Therefore, the manifold $(G,\f,\allowbreak{}\xi,\allowbreak{}\eta,\allowbreak{}g)$
admits a para-Ricci-like soliton with potential $\xi$.

\begin{remark}
The constructed 5-dimensional para-Sasaki-like manifold $(G,\allowbreak{}\f,\allowbreak{}\xi,\eta,g)$ is $\eta$-Einstein with constants given in \eqref{abcS} and $(G,\f,\allowbreak{}\xi,\eta,g)$ admits a para-Ricci-like soliton with potential $\xi$ and constants determined in \eqref{lmnS}. Moreover, these constants satisfy \eqref{SlElRl-const} and therefore this example supports the assertion (iii) of Theorem A.
\end{remark}

\subsection{Example 2}
In \cite{ManTav2}, {an example of a 3-dimensional apapR manifold belonging to the basic class $\F_5$} is considered. The constructed manifold is denoted by $(L,\f,\allowbreak{}\xi,\eta,g)$, where $L$ is a Lie group with a basis $\{e_0, e_1, e_2\}$ of left invariant vector fields.
The corresponding Lie algebra and the apapR structure are defined by:
\begin{equation}\label{strL}
\begin{array}{c}
[e_0,e_1]=p e_1, \qquad [e_0,e_2]=p e_2, \qquad
[e_1,e_2]=0,\qquad p\in\R,\\
\begin{array}{l}
\f e_0=0,\qquad \f e_1=e_{2},\qquad \f e_{2}= e_1,\qquad \xi=
e_0,\\
\eta(e_0)=1,\qquad \eta(e_1)=\eta(e_{2})=0,
\end{array}\\
\begin{array}{l}
  g(e_0,e_0)=g(e_1,e_1)=g(e_{2},e_{2})=1, \\
  g(e_0,e_1)=g(e_0,e_2)=g(e_1,e_2)=0.
\end{array}
\end{array}
\end{equation}

In the {same} paper, it is obtained that $(L,\f,\allowbreak{}\xi,\eta,g)$ is Einstein with constants
\begin{equation}\label{abcF5}
(a,b,c)=(-2p^2,0,0).
\end{equation}

Moreover, it is easy to check that $\xi$ is a torse-forming vector field with constant
\begin{equation}\label{fF5}
f=-p.
\end{equation}

By virtue of \eqref{strL} and \eqref{Lg=nxi},
we obtain the components $\left(\LL_\xi g\right)_{ij}$ and
the non-zero of them are
\begin{equation}\label{Lex2}
\begin{array}{l}
\left(\LL_\xi g\right)_{11}=\left(\LL_\xi g\right)_{22}=-2p.
\end{array}
\end{equation}

After that, using \eqref{strL}  and \eqref{Lex2}, we get that \eqref{defRl} is satisfied for
\begin{equation}\label{lmnF5}
(\lm,\mu,\nu)=(p+2p^2,0,-p).
\end{equation}
Therefore, the manifold $(L,\f,\allowbreak{}\xi,\allowbreak{}\eta,\allowbreak{}g)$
admits a para-Ricci-like soliton with torse-forming potential $\xi$.

\begin{remark}
The constructed 3-dimensional apapR manifold $(L,\f,\allowbreak{}\xi,\eta,g)$ is Einstein with constants given in \eqref{abcF5} and $(L,\f,\allowbreak{}\xi,\eta,g)$ admits an $\eta$-Ricci soliton with potential $\xi$ and constants determined in \eqref{fF5} and \eqref{lmnF5}. Moreover, these constants satisfy \eqref{tfElRl-const} and therefore this example supports the assertion (iii) of Theorem B.
\end{remark}

\subsection*{Acknowledgment}
The research of H.M. is partially supported by the National Scientific Program ``Young Researchers and Post-Doctorants'' and the project MU21-FMI-008 of the Scientific Research Fund, University of Plovdiv Paisii Hilendarski. The research of M.M. is partially supported by projects MU21-FMI-008 and FP21-FMI-002 of the Scientific Research Fund, University of Plovdiv Paisii Hilendarski.


\begin{thebibliography}{99}

\bibitem{AyaYil}
G. Ayar, M. Y\i{}ld\i{}r\i{}m,
$\eta$-Ricci solitons on nearly Kenmotsu manifolds,
 {\it Asian-Eur. J. Math.} \textbf{12} (6) (2019) 2040002 (8 pages).


\bibitem{BagIng12}
C.\,S. Bagewadi, G. Ingalahalli,
Ricci solitons in Lorentzian $\al$-Sasakian manifolds,
{\it Acta Math.} \textbf{28} (2012) 59--68.


\bibitem{BagIng13}
C.\,S. Bagewadi, G. Ingalahalli,
Certain results on Ricci solitons in trans-Sasakian manifolds,
{\it J. Math.} \textbf{2013} art. ID 787408 (10 pages).



\bibitem{Bla15}
A.\,M. Blaga,
$\eta$-Ricci solitons on para-Kenmotsu manifolds,
{\it Balkan J. Geom. Appl.} \textbf{20} (2015) 1--13.




\bibitem{BlaCra}
A.\,M. Blaga, M. Crasmareanu,
Torse-forming $\eta$-Ricci solitons in almost paracontact $\eta$-Ein\-stein geometry,
{\it Filomat}  \textbf{31} (2) (2017), 499--504.


\bibitem{BlaPer}
A.\,M. Blaga, S.\,Y. Perkta\c{s},
Remarks on almost $\eta$-Ricci solitons in $(\varepsilon)$-para Sasakian manifolds,
{\it Commun. Fac. Sci. Univ. Ank. Ser. A1 Math. Stat.}
\textbf{68} (2) (2019) 1621--1628.



\bibitem{Bro-etal}
M. Brozos-Vazquez, G. Calvaruso, E. Garcia-Rio, S. Gavino-Fernandez,
Three-dimensional Lorentzian homogeneous Ricci solitons,
{\it Israel J. Math.} \textbf{188} (2012)  385--403.




\bibitem{CalPer}
G. Calvaruso, D. Perrone,
Geometry of $H$-paracontact metric manifolds,
{\it Publ. Math. Debrecen} \textbf{86} (3-4) (2015) 325--346.

\bibitem{Cao}
H.-D. Cao,
Recent progress on Ricci solitons,
{\it Adv. Lect. Math. (ALM)} \textbf{11} (2009) 1--38.

\bibitem{ChoKim}
J.\,T. Cho, M. Kimura,
Ricci solitons and real hypersurfaces in a complex space form,
{\it T\^{o}hoku Math. J.} \textbf{61} (2) (2009) 205--212.


\bibitem{ChoLuNi}
B. Chow, P. Lu, L. Ni,
Hamilton's Ricci Flow,
Graduate Studies in Mathematics \textbf{77} (2006),
AMS Sci. Press, Providence, RI, USA.


\bibitem{GalCra}
C. G\u{a}lin, M. Crasmareanu,
From the Eisenhart problem to Ricci solitons $f$-Kenmotsu manifolds
{\it Bull. Malays. Math. Sci. Soc.} \textbf{33} (3) (2010) 361--368.




\bibitem{Ham82}
R.\,S. Hamilton,
Three-manifolds with positive Ricci curvature,
{\it J. Differential Geom.} \textbf{ 17} (1982) 255--306.




%



\bibitem{IngBag}
G. Ingalahalli, C.\,S. Bagewadi,
Ricci solitons in $\al$-Sasakian manifolds,
Int. Sch. Res. Notices Geometry \textbf{2012} art. ID 421384 (2012) (13 pages).


\bibitem{IvMaMa2}
S. Ivanov, H. Manev, M. Manev,
Para-Sasaki-like Riemannian manifolds and new Einstein metrics,
{\it Rev. R. Acad. Cienc. Exactas Fis. Nat. Ser. A Math. RACSAM}  \textbf{115} (2021) 112.





\bibitem{MM-Sol1}
M. Manev
Ricci-like solitons on almost contact
B-metric manifolds
{\it J. Geom. Phys.} \textbf{154} (2020) 103734.





\bibitem{ManSta}
M. Manev, M. Staikova,
On almost paracontact Riemannian manifolds of type $(n, n)$
{\it J. Geom.} \textbf{72} (2001) 108--114.


\bibitem{ManTav57}
M. Manev, V. Tavkova,
On almost paracontact almost paracomplex Riemannian
manifolds {\it Facta Univ. Ser. Math. Inform.} \textbf{33} (2018) 637--657.

\bibitem{ManTav2}
M. Manev, V. Tavkova,
Lie groups as 3-dimensional almost paracontact almost paracomplex Riemannian manifolds
{\it J. Geom.} \textbf{110} (2019) 43.




\bibitem{NagPre}
H.\,G. Nagaraja, C.\,R. Premalatha,
Ricci solitons in Kenmotsu manifolds,
{\it J. Math. Anal.} \textbf{3} (2) (2012) 18--24.


\bibitem{PraHad}
D.\,G. Prakasha, B.\,S. Hadimani,
$\eta$-Ricci solitons on para-Sasakian manifolds,
{\it J. Geom.} \textbf{108} (2) (2017) 383--392.




\bibitem{Sato76}
I. Sat\={o},
On a structure similar to the almost contact structure,
{\it Tensor (N.S.)} \textbf{30} (1976) 219--224.


\bibitem{Shar}
R. Sharma,
Certain results on K-contact and $(\kappa,\mu)$-contact manifolds,
{\it J. Geom.} \textbf{89} (1-2) (2008) 138--147.

%




\end{thebibliography}
\end{document}